\documentclass[12pt]{amsart}
\usepackage[inner=3.25cm,outer=3.25cm, top=3cm, bottom=2.5cm]{geometry}

\usepackage{amsmath,amssymb}
\usepackage[utf8,utf8x]{inputenc}
\usepackage[pagebackref=true]{hyperref}

\usepackage{todonotes}

\usepackage{caption}
\usepackage{tikz}
\usetikzlibrary{calc}
\usetikzlibrary{intersections}

\newcommand{\NN }{\mathbb{N}}

\newcommand{\ZZ }{\mathbb{Z}}
\newcommand{\KK }{\mathbb{K}}

\newcommand{\Ac }{\mathcal{A}}

\newcommand{\Bc }{\mathcal{B}}
\newcommand{\Dc }{\mathcal{D}}

\newcommand{\Cc }{\mathcal{C}}

\newcommand\MATF{{\mathfrak{MF}}}
\newcommand\IFC{{\mathfrak{IF}}}
\newcommand\CCl{{\mathfrak{C}}}
\newcommand\AFC{{\mathfrak{AF}}}

\DeclareMathOperator{\Der}{Der}
\DeclareMathOperator{\pdeg}{pdeg}

\newcommand\ddxi[1]{\partial/\partial x_{#1}}

\DeclareMathOperator{\rk}{rk}

\numberwithin{equation}{section}

\theoremstyle{plain}
\newtheorem{lemma}[equation]{Lemma}
\newtheorem{theorem}[equation]{Theorem}

\newtheorem{corollary}[equation]{Corollary}
\newtheorem{proposition}[equation]{Proposition}
\theoremstyle{definition}
\newtheorem{defn}[equation]{Definition}
\newtheorem{remark}[equation]{Remark}

\newtheorem{example}[equation]{Example}
\newtheorem{problem}[equation]{Problem}
\newtheorem{question}[equation]{Question}

\title[MAT-free reflection arrangements]
{MAT-free reflection arrangements}

\author{M.~Cuntz}
\address{Michael Cuntz,
Institut für Algebra, Zahlentheorie und Diskrete Mathematik,
Fakult\"at für Mathematik und Physik,
Leibniz Universit\"at Hannover,
Welfengarten 1,
D-30167 Hannover, Germany}
\email{cuntz@math.uni-hannover.de}

\author{P.~M\"ucksch}
\address{Paul M\"ucksch,
Fakult\"at f\"ur Mathematik, Ruhr-Universit\"at Bochum,
D-44780 Bochum, Germany}
\email{paul.muecksch@rub.de}

\begin{document}


\begin{abstract}
We introduce the class of MAT-free hyperplane arrangements which is based on the Multiple Addition Theorem by Abe, Barakat, 
Cuntz, Hoge, and Terao. We also investigate the closely related class of MAT2-free arrangements based on a recent generalization
of the Multiple Addition Theorem by Abe and Terao.
We give classifications of the irreducible complex reflection arrangements which are MAT-free respectively MAT2-free.
Furthermore, we ask some questions concerning relations to other classes of free arrangements. 
\end{abstract}

\maketitle


\section{Introduction}

A hyperplane arrangement $\Ac$ is a finite set of hyperplanes in a finite dimensional vector space $V \cong \KK^\ell$.
The intersection lattice $L(\Ac)$ of $\Ac$ encodes its combinatorial properties.
It is a main theme in the study of hyperplane arrangements to link algebraic properties of $\Ac$ with the combinatorics of $L(\Ac)$.

The algebraic property of \emph{freeness} of a hyperplane arrangement $\Ac$ was first studied by Saito \cite{Saito80_LogForms} and
Terao \cite{Terao1980_FreeI}.
In fact, it turns out that freeness of $\Ac$ imposes strong combinatorial constraints on $L(\Ac)$:
by Terao's Factorization Theorem \cite[Thm.~4.137]{OrTer92_Arr} its characteristic polynomial factors over the integers.
Conversely, sufficiently strong conditions on $L(\Ac)$ imply the freeness of $\Ac$.
One of the main tools to derive such conditions is Terao's Addition-Deletion Theorem \ref{Thm_Add_Del}.
It motivates the class of \emph{inductively free} arrangements (see Definition \ref{Def_IF}). 
In this class the freeness of $\Ac$ is combinatorial, i.e.\ it is completely determined by $L(\Ac)$ (cf.\ Definition \ref{Def_CombClass}).
Recently, a remarkable generalization of the Addition-Deletion theorem
was obtained by Abe. His Division Theorem \cite[Thm.~1.1]{Abe16_DivFree}
motivates the class of \emph{divisionally free} arrangements. In this class freeness is a combinatorial property too.

Despite having these useful tools at hand, it is still a major open problem, 
known as Terao's Conjecture, whether in general the freeness of $\Ac$ actually depends only on $L(\Ac)$, 
provided the field $\KK$ is fixed 
(see \cite{Zieg90_MatrFree} for a counterexample when one fixes $L(\Ac)$ but changes the field).
We should also mention at this point the very recent results by Abe further examining Addition-Deletion constructions
together with divisional freeness \cite{Abe2018_DelThm_Combinatorics}, \cite{Abe18_AddDel_Combinatorics}.

A variation of the addition part of the Addition-Deletion theorem \ref{Thm_Add_Del} 
was obtained by Abe, Barakat, Cuntz, Hoge, and Terao
in \cite{ABCHT16_FreeIdealWeyl}: the Multiple Addition Theorem \ref{Thm_MAT} (MAT for short). Using this theorem, the authors gave 
a new uniform proof of the Kostant-Macdonald-Shapiro-Steinberg formula for the exponents of a Weyl group.
In the same way the Addition-Theorem defines the class of inductively free arrangements,
it is now natural to consider the class $\MATF$ of those free arrangements, called \emph{MAT-free}, which can be build inductively using the MAT (Definition \ref{Def_MATfree}).
It is not hard to see (Lemma \ref{Lem_MATComb}) that MAT-freeness only depends on $L(\Ac)$.
In this paper, we investigate classes of MAT-free arrangements beyond the classes considered in \cite{ABCHT16_FreeIdealWeyl}.

Complex reflection groups (classified by Shephard and Todd \cite{ST_1954_fcrg}) play an important role in the study of hyperplane arrangements:
many interesting examples and counterexamples are related or derived from the reflection arrangement $\Ac(W)$ of a complex reflection group $W$.
It was proven by Terao \cite{Terao1980_FreeUniRefArr} that reflection arrangements are always free.
There has been a series of investigations dealing with reflection arrangements and their connection to the aforementioned
combinatorial classes of free arrangements (e.g.\ \cite{BC12_CoxCrystIndFree}, \cite{HoRoe15_IndFreeRef}, \cite{Abe16_DivFree}).
Therefore, it is natural to study reflection arrangements in conjunction with the new class of MAT-free arrangements.

Our main result is the following.
\begin{theorem}\label{Thm_matref}
Except for the arrangement $\Ac(G_{32})$, an irreducible reflection arrangement is MAT-free if and only if it is inductively free.
The arrangement $\Ac(G_{32})$ is inductively free but not MAT-free. Thus every reflection arrangement is MAT-free 
except the reflection arrangements of the imprimitive reflection groups $G(e,e,\ell)$, $e>2$, $\ell>2$ and of the reflection groups
\[ G_{24}, G_{27}, G_{29}, G_{31}, G_{32}, G_{33}, G_{34}. \]
\end{theorem}

A further generalization of the MAT \ref{Thm_MAT} was very recently obtained by
Abe and Terao \cite{AbeTer18_MultAddDelRes}: the Multiple Addition Theorem 2 \ref{Thm_MAT2} (MAT2 for short).
Again, one might consider the inductively defined class of arrangements 
which can be build from the empty arrangement using this more general tool,
i.e.\ the class $\MATF'$ of \emph{MAT2-free} arrangements (Defintion \ref{Def_MAT2free}).
By definition, this class contains the class of MAT-free arrangements.
Regarding reflection arrangements we have the following:
\begin{theorem}\label{Thm_mat2free}
Let $\Ac = \Ac(W)$ be an irreducible reflection arrangement. 
Then $\Ac$ is MAT2-free if and only if it is MAT-free.
\end{theorem}

In contrast to (irreducible) reflection arrangements, in general the class of MAT-free arrangements 
is properly contained in the class of MAT2-free arrangements (see Proposition \ref{Propo_MATproperSubclassMAT2}).

Based on our classification of MAT-free (MAT2-free) reflection arrangements and other known examples (\cite{ABCHT16_FreeIdealWeyl}, \cite{CRS17_IdealIF})
we arrive at the following question:
\begin{question}\label{Ques_MATIF}
Is every MAT-free (MAT2-free) arrangement inductively free?
\end{question}
In \cite{CRS17_IdealIF} the authors proved that all ideal subarrangements of a Weyl arrangement are inductively free
by extensive computer calculations. 
A positive answer to Question \ref{Ques_MATIF} would directly imply their result and
yield a uniform proof (cf.\ \cite[Rem.~1.5(d)]{CRS17_IdealIF}).

Looking at the class of divisionally free arrangements which properly contains the class of inductively free arrangements
\cite[Thm.~4.4]{Abe16_DivFree} a further natural question is:
\begin{question}\label{Ques_MATDF}
Is every MAT-free (MAT2-free) arrangement divisionally free?
\end{question}

This article is organized as follows: in Section \ref{Sec_freeArr} we briefly recall some notions and results 
about hyperplane arrangements and free arrangements used throughout our exposition.
In Section \ref{Sec_MAT} we give an alternative characterization of MAT-freeness and two easy necessary conditions for MAT/MAT2-freeness.
Furthermore, we comment on the relation of the two classes $\MATF$ and $\MATF'$ and on the product construction.
Section \ref{Sec_ProofImprim} and Section \ref{Sec_ProofPrim} contain the proofs of Theorem \ref{Thm_matref}
and Theorem \ref{Thm_mat2free}.
In the last Section \ref{Sec_Remarks} we comment on Question \ref{Ques_MATIF} and further problems
connected with MAT-freeness.

\subsection*{Acknowledgments}
We thank Gerhard R{\"o}hrle for valuable comments on an earlier draft of our manuscript.


\section{Hyperplane arrangements and free arrangements}\label{Sec_freeArr}

Let $\Ac$ be a hyperplane arrangement in $V \cong \KK^\ell$. If $\Ac$ is empty, then it
is denoted by $\Phi_\ell$.

The \emph{intersection lattice} $L(\Ac)$ of $\Ac$ consists of all intersections of elements of $\Ac$
including $V$ as the empty intersection. 
Indeed, with the partial order by reverse inclusion $L(\Ac)$ is a geometric lattice \cite[Lem.~2.3]{OrTer92_Arr}.
The \emph{rank} $\rk(\Ac)$ of $\Ac$ is defined as the codimension of the intersection of all hyperplanes in $\Ac$.

If $x_1,\ldots,x_\ell$ is a basis of $V^*$, to explicitly give a hyperplane 
we use the notation $(a_1,\ldots,a_\ell)^\perp := \ker(a_1x_1+\ldots+a_\ell x_\ell)$.

\begin{defn}\label{Def_CombClass}
Let $\CCl$ be a class of arrangements and let $\Ac \in \CCl$.
If for all arrangements $\Bc$ with $L(\Bc) \cong L(\Ac)$, (where $\Ac$ and $\Bc$ do not have to be defined over the same
field), we have $\Bc \in \CCl$, then the class $\CCl$ is called \emph{combinatorial}.

If $\CCl$ is a combinatorial class of arrangements such that every arrangement in $\CCl$ is free than
$\Ac \in \CCl$ is called \emph{combinatorially free}.
\end{defn}

For $X \in L(\Ac)$ the \emph{localization} $\Ac_X$ of $\Ac$ at X is defined by:
\[
\Ac_X := \{ H \in \Ac \mid X \subseteq H \},
\]
and the \emph{restriction} $\Ac^X$ of $\Ac$ to $X$ is defined by: 
$$
\Ac^X := \{ X\cap H \mid H \in \Ac \setminus \Ac_X \}.
$$

Let $\Ac_1$ and $\Ac_2$ be two arrangements in $V_1$ respectively
$V_2$. Then their \emph{product} $\Ac_1 \times \Ac_2$ is defined as the
arrangement in $V = V_1 \oplus V_2$ consisting of the following hyperplanes:
\[
\Ac_1 \times \Ac_2 := \{H_1 \oplus V_2 \mid H_1 \in \Ac_1 \} \cup \{V_1 \oplus H_2 \mid H_2 \in \Ac_2 \}.
\]
We note the following facts about products (cf.\ \cite[Ch.~2]{OrTer92_Arr}):
\begin{itemize}
\item $| \Ac_1\times\Ac_2| = |\Ac_1| + |\Ac_2|$.
\item $L(\Ac_1\times\Ac_2) = \{X_1 \oplus X_2 \mid X_1 \in L(\Ac_1)$ and $X_2 \in L(\Ac_2)\}$.
\item $(\Ac_1\times\Ac_2)^X = \Ac_1^{X_1} \times \Ac_2^{X_2}$ if $X = X_1\oplus X_2$ with $X_i \in L(\Ac_i)$.
\end{itemize} 

Let $S = S(V^*)$ be the symmetric algebra of the dual space.
We fix a basis $x_1,\ldots,x_\ell$ for $V^*$ and identify $S$ with the polynomial ring $\KK[x_1,\ldots,x_\ell]$.
The algebra $S$ is equipped with the grading by polynomial degree: $S = \bigoplus_{p\in \ZZ} S_p$,
where $S_p$ is the set of homogeneous polynomials of degree $p$ ($S_p = \{0\}$ for $p < 0$).

A $\KK$-linear map $\theta:S\to S$ which satisfies $\theta(fg) = \theta(f)g + f\theta(g)$ is called a $\KK$-\emph{derivation}.
Let $\Der(S)$ be the $S$-module of $\KK$-derivations of $S$. It is a free $S$-module with basis
$D_1,\ldots,D_\ell$ where $D_i$ is the partial derivation $\ddxi{i}$.
We say that $\theta \in \Der(S)$ is \emph{homogeneous of polynomial degree} $p$ provided
$\theta = \sum_{i=1}^\ell f_i D_i$ with $f_i \in S_p$ for each $1 \leq i \leq \ell$.
In this case we write $\pdeg{\theta} = p$.
We obtain a $\ZZ$-grading for the $S$-module $\Der(S)$:
$\Der(S) = \bigoplus_{p \in \ZZ} \Der(S)_p$.

\begin{defn}
For $H \in \Ac$ we fix $\alpha_H \in V^*$ with $H = \ker(\alpha_H)$.
The \emph{module of $\Ac$-derivations} is defined by
\begin{equation*}
  D(\Ac) := \{ \theta \in \Der(S) \mid \theta(\alpha_H) \in {\alpha_H}S \text{ for all } H \in \Ac\}.
\end{equation*}
We say that $\Ac$ is \emph{free} if the module of $\Ac$-derivations is a free $S$-module.
\end{defn}

If $\Ac$ is a free arrangement we may choose a homogeneous basis $\{ \theta_1, \ldots, \theta_\ell \}$ for $D(\Ac)$.
Then the polynomial degrees of the $\theta_i$ are called the \emph{exponents} of $\Ac$ and
they are uniquely determined by $\Ac$, \cite[Def.\ 4.25]{OrTer92_Arr}.
We write $\exp(\Ac) := (\pdeg{\theta_1},\ldots$, $\pdeg{\theta_\ell})$.
Note that the empty arrangement $\Phi_\ell$ is free with $\exp(\Phi_\ell)=(0,\ldots,0)\in\ZZ^\ell$. 
If $d_1,\ldots,d_\ell \in \ZZ$ with $d_1 \leq d_2 \leq \ldots \leq d_\ell$ we write $(d_1,\ldots,d_\ell)_\leq$.

The notion of freeness is compatible with products of arrangements:
\begin{proposition}[{\cite[Prop.~4.28]{OrTer92_Arr}}]\label{Prop_ProdFree}
Let $\Ac=\Ac_1\times \Ac_2$ be a product of two arrangements.
Then $\Ac$ is free if and only if both $\Ac_1$ and $\Ac_2$ are free. In this case
if $\exp(\Ac_i)=(d^i_1,\ldots,d^i_{\ell_i})$ for $i=1,2$ then 
\[
\exp(\Ac) = (d^1_1,\ldots,d^1_{\ell_1},d^2_1,\ldots,d^2_{\ell_2}).
\]
\end{proposition}

The following theorem provides a useful tool to prove the freeness of arrangements.

\begin{theorem}[Addition-Deletion {\cite[Thm.\ 4.51]{OrTer92_Arr}}]\label{Thm_Add_Del}
Let $\Ac$ be a hyperplane arrangement and $H_0 \in \Ac$.
We call $(\Ac,\Ac'=\Ac\setminus \{H_0\},\Ac''=\Ac^{H_0})$ a \emph{triple of arrangements}.
Any two of the following statements imply the third:
\begin{enumerate}
\item $\Ac$ is free with $\exp(\Ac) = (b_1,\ldots,b_{l-1},b_\ell)$, 
\item $\Ac'$ is free with $\exp(\Ac') = (b_1,\ldots,b_{\ell-1},b_\ell-1)$, 
\item $\Ac''$ is free with $\exp(\Ac'') = (b_1,\ldots,b_{\ell-1})$.
\end{enumerate}
\end{theorem}

The preceding theorem motivates the following definition.

\begin{defn}[{\cite[Def.\ 4.53]{OrTer92_Arr}}] \label{Def_IF}
The class $\IFC$ of \emph{inductively free} arrangements is the smallest class of arrangements which satisfies
\begin{enumerate}
\item the empty arrangement $\Phi_ \ell$ of rank $\ell$ is in $\IFC$ for $\ell \geq 0$,
\item if there exists a hyperplane $H_0 \in \Ac$ such that $\Ac'' \in \IFC$, $\Ac' \in \IFC$, and
$\exp(\Ac'') \subset \exp(\Ac')$, then $\Ac$ also belongs to $\IFC$.
\end{enumerate}
Here $(\Ac,\Ac',\Ac'') = (\Ac,\Ac\setminus \{H_0\},\Ac^{H_0})$ is a triple as in Theorem \ref{Thm_Add_Del}.
\end{defn}
The class $\IFC$ is easily seen to be combinatorial \cite[Lem.~2.5]{CHo15_FreeNotRecFree}.

The following result was a major step in the investigation of freeness properties for reflection arrangements.
\begin{theorem}[{\cite[Thm.~1.1]{HoRoe15_IndFreeRef}, \cite[Thm.~5.14]{BC12_CoxCrystIndFree}}]\label{Thm_ClassIFReflArr}
For W a finite complex reflection group,
the reflection arrangement $\Ac(W)$ is inductively free if and only if $W$ does not admit an
irreducible factor isomorphic to a monomial group $G(r,r,\ell)$ for $r, \ell ≥ 3$, $G_{24}$, $G_{27}$, $G_{29}$,
$G_{31}$, $G_{33}$, or $G_{34}$.
\end{theorem}

\begin{defn}[cf.\ \cite{AT15_FreeFilt}]
Let $\Ac$ be an arrangement with $|\Ac|=n$. We say that $\Ac$ has a \emph{free filtration} if there are subarrangements
$$\emptyset = \Ac_0 \subsetneq \Ac_1 \subsetneq \cdots \subsetneq \Ac_{n-1} \subsetneq \Ac_n = \Ac$$
such that $|\Ac_i| = i$ and $\Ac_i$ is free for all $1 \leq i \leq n$.
\end{defn}

Very recently, Abe \cite{Abe18_AddDel_Combinatorics} introduced the class $\AFC$ of \emph{additionally free} arrangements.
Arrangements in $\AFC$ are by definition exactly the arrangements admitting a free filtration.
Furthermore, it is a direct consequence of \cite[Thm.~1.4]{Abe18_AddDel_Combinatorics} that the class $\AFC$ is combinatorial.


\section{Multiple Addition Theorem}\label{Sec_MAT}

The following theorem presented in \cite{ABCHT16_FreeIdealWeyl} is a variant of the addition part ((2) and (3) imply (1)) of Theorem \ref{Thm_Add_Del}.

\begin{theorem}[Multiple Addition Theorem (MAT)]\label{Thm_MAT}
Let $\Ac'$ be a free arrangement with
$\exp(\Ac')=(d_1,\ldots,d_\ell)_\le$
and $1 \le p \le \ell$ the multiplicity of the highest exponent, i.e.,
$$ d_{\ell-p} < d_{\ell-p+1} =\cdots=d_\ell=:d. $$
Let $H_1,\ldots,H_q$ be hyperplanes with
$H_i \not \in \Ac'$ for $i=1,\ldots,q$. Define
$$ \Ac''_j:=(\Ac'\cup \{H_j\})^{H_j}=\{H\cap H_{j} \mid H\in \Ac'\}, \quad j=1,\ldots,q.$$
Assume that the following three conditions are satisfied:
\begin{itemize}
\item[(1)]
$X:=H_1 \cap \cdots \cap H_q$ is $q$-codimensional.
\item[(2)]
$X \not \subseteq \bigcup_{H \in \Ac'} H$.
\item[(3)]
$|\Ac'|-|\Ac''_j|=d$ for $1 \le j \le q$.
\end{itemize}
Then $q \leq p$ and $\Ac:=\Ac' \cup \{H_1,\ldots,H_q\}$ is free with
$\exp(\Ac)=(d_1,\ldots,d_{\ell-q},d+1,$ $\ldots,d+1)_\le$.
\label{MAT}
\end{theorem}

Note that in contrast to Theorem \ref{Thm_Add_Del} no freeness condition on the restriction is needed
to conclude the freeness of $\Ac$ in Theorem \ref{Thm_MAT}.
The MAT motivates the following definition.

\begin{defn}
\label{Def_MATfree}
The class $\MATF$ of \emph{MAT-free} arrangements is the smallest class of arrangements subject to
\begin{itemize}
\item[(i)] $\Phi_\ell$ belongs to $\MATF$, for every $\ell \ge 0$;
\item[(ii)]
if $\Ac' \in \MATF$ with
$\exp(\Ac')=(d_1,\ldots,d_\ell)_\le$
and $1 \le p \le \ell$ the multiplicity of the highest exponent $d=d_\ell$, and
if $H_1,\ldots,H_q$, $q\le p$ are hyperplanes with
$H_i \not \in \Ac'$ for $i=1,\ldots,q$ such that:
\begin{itemize}
\item[(1)]
$X:=H_1 \cap \cdots \cap H_q$ is $q$-codimensional,
\item[(2)]
$X \not \subseteq \bigcup_{H \in \Ac'} H$,
\item[(3)]
$|\Ac'|-|(\Ac'\cup \{H_j\})^{H_j}|=d$, for $1 \le j \le q$,
\end{itemize}
then $\Ac:=\Ac' \cup \{H_1,\ldots,H_q\}$ also belongs to $\MATF$
and has exponents
$\exp(\Ac) = (d_1,\ldots,d_{\ell-q},d+1,\ldots,d+1)_\le$.
\end{itemize}
\end{defn}

Abe and Terao \cite{AbeTer18_MultAddDelRes} proved the following generalization of Theorem \ref{Thm_MAT}:

\begin{theorem}[Multiple Addition Theorem 2 (MAT2), {\cite[Thm.~1.4]{AbeTer18_MultAddDelRes}}]\label{Thm_MAT2}
Assume that $\Ac'$ is a free arrangement with 
$\exp(\Ac')=(d_1,d_2,\ldots,d_\ell)_\le$. 
Let 
\[
t :=
\begin{cases} 
   \min\{i \mid d_i \neq 0\} & \text{if } \Ac'\neq \Phi_\ell \\
   0 & \text{if } \Ac'=\Phi_\ell
\end{cases}.
\]
For $H_s,\ldots,H_\ell \notin \Ac$ with $s > t$, 
define $\Ac_j'':=(\Ac' \cup \{H_j\})^{H_j}$, $\Ac:=\Ac' \cup \{H_s,\ldots,H_\ell\}$ 
and assume the following conditions:
\begin{itemize}
\item[(1)]
$X:=\bigcap_{i=s}^\ell H_i$ is $(\ell-s+1)$-codimensional,
\item[(2)]
$X \not \subset \bigcup_{K \in \Ac'}K$, and 
\item[(3)]
$|\Ac'| -|\Ac_j''|=d_j$ for $j=s,\ldots,\ell$.
\end{itemize}
Then $\Ac$ is free with exponents 
$(d_1,d_2,\ldots,d_{s-1},d_{s}+1,\ldots,d_\ell+1)_\leq$. 
Moreover, there is a basis 
$\theta_1,\theta_2,\ldots,\theta_{s-1},\eta_s,\ldots,\eta_\ell$ for $D(\Ac')$ such that 
$\deg \theta_i=d_i,\ \deg \eta_j=d_j$, 
$\theta_i \in D(\Ac)$ and $\eta_j \in D(\Ac \setminus \{H_j\})$ for all $i$ and $j$.
\end{theorem}

This in turn motivates:
\begin{defn}
\label{Def_MAT2free}
The class $\MATF'$ of \emph{MAT2-free} arrangements is the smallest class of arrangements subject to
\begin{itemize}
\item[(i)] $\Phi_\ell$ belongs to $\MATF'$, for every $\ell \ge 0$;
\item[(ii)]
if $\Ac' \in \MATF'$ with
$\exp(\Ac')=(d_1,d_2,\ldots,d_\ell)_\le$
and if $H_s,\ldots,H_\ell$ are hyperplanes with
$H_i \not \in \Ac'$ for $i=s,\ldots,\ell$, where 
\[
s > 
\begin{cases} 
   \min\{i \mid d_i \neq 0\} & \text{if } \Ac'\neq \Phi_\ell \\
   0 & \text{if } \Ac'=\Phi_\ell
\end{cases},
\]
and with
\begin{itemize}
\item[(1)]
$X:=H_s \cap \cdots \cap H_\ell$ is $(\ell-s+1)$-codimensional,
\item[(2)]
$X \not \subseteq \bigcup_{H \in \Ac'} H$,
\item[(3)]
$|\Ac'|-|(\Ac'\cup \{H_j\})^{H_j}|=d_j$ for $s \le j \le \ell$,
\end{itemize}
then $\Ac:=\Ac' \cup \{H_s,\ldots,H_\ell\}$ also belongs to $\MATF'$
and has exponents
$\exp(\Ac) = (d_1,\ldots,d_{s-1},d_s+1,\ldots,d_\ell+1)_\le$.
\end{itemize}
\end{defn}

We note the following:
\begin{remark}\label{Rem_MATexpMAT2}
\begin{enumerate}
\item We have $\MATF \subseteq \MATF'$.
\item If $\Ac$ is a free arrangement with $\exp(\Ac) = (0,\ldots,0,1,\ldots,1,d,\ldots,d)_\le$, i.e.\
$\Ac$ has only two distinct exponents $\neq 0$, then it is clear from the definitions that $\Ac$ is MAT2-free
if and only if $\Ac$ is MAT-free.
\end{enumerate} 
\end{remark}

\begin{example}\label{Exam_rk2_boolean_WeylMAT}
\begin{enumerate}
\item If $\rk(\Ac)=2$ then $\Ac$ is MAT-free and therefore MAT2-free too.
\item Every ideal subarrangement of a Weyl arrangement is MAT-free and therefore also MAT2-free, \cite{ABCHT16_FreeIdealWeyl}.
\end{enumerate}
\end{example}

\begin{lemma}\label{Lem_MATComb}
The classes $\MATF$ and $\MATF'$ are combinatorial.
\end{lemma}
\begin{proof}
The class of all empty arrangements is combinatorial and contained in $\MATF$.
Let $\Ac \in \MATF$ ($\Ac \in \MATF'$). 
Since conditions (1)--(3) in Defintion \ref{Def_MATfree} (respectively Defintion \ref{Def_MAT2free})
only depend on $L(\Ac)$ the claim follows. 
See also \cite[Thm.~5.1]{AbeTer18_MultAddDelRes}.
\end{proof}

If an arrangement $\Ac$ is MAT-free, the MAT-steps yield a partition of $\Ac$ whose dual partition gives the exponents of $\Ac$.
Vice versa, the existence of such a partition suffices for the MAT-freeness of the arrangement:

\begin{lemma}\label{Lem_MATPart}
Let $\Ac$ be an $\ell$-arrangement. Then $\Ac$ is MAT-free if and only if there exists a partition
$\pi = (\pi_1|\cdots|\pi_n)$ of $\Ac$ where for all $0 \leq k \leq n-1$,
\begin{itemize}
\item[(1)] $\rk(\pi_{k+1}) = \vert \pi_{k+1} \vert$,
\item[(2)] $\cap_{H \in \pi_{k+1}} H = X_{k+1} \nsubseteq \bigcup_{H' \in \Ac_k}H'$ where $\Ac_k = \bigcup_{i=1}^k \pi_i$,
\item[(3)] $\vert \Ac_k \vert - \vert (\Ac_k \cup \{H\})^H \vert = k$ for all $H \in \pi_{k+1}$.
\end{itemize}
In this case $\Ac$ has exponents $\exp(\Ac) = (d_1,\ldots,d_\ell)_\le$ with $d_i = |\{k \mid |\pi_k|\geq \ell-i+1 \}|$.
\end{lemma}
\begin{proof}
This is immediate from the definition.
\end{proof}

\begin{defn}
If $\pi$ is a partition as in Lemma \ref{Lem_MATPart} then $\pi$ is called an \emph{MAT-partition} for $\Ac$.

If we have chosen a linear ordering $\Ac = \{ H_1,\ldots,H_m\}$ of the hyperplanes in $\Ac$, to specify the partition $\pi$,
we give the corresponding ordered set partition of $[m] = \{1,\ldots,m\}$.
\end{defn}

\begin{example}\label{Exam_IFbnMAT2}
Supersolvable arrangements, a proper subclass of inductively free arrangements \cite[Thm.~4.58]{OrTer92_Arr}, 
are not necessarily MAT2-free: an easy calculation shows that the arrangement denoted $\Ac(10,1)$ 
in \cite{Grue09_SimplArr} is supersolvable but not MAT2-free. 
In particular $\Ac(10,1)$ is neither MAT-free.
\end{example}

Restrictions of MAT2-free (MAT-free) arrangements are not necessarily MAT2-free (MAT-free):
\begin{example}\label{Exam_ResNotMAT2}
Let $\Ac = \Ac(E_6)$ be the Weyl arrangement of the Weyl group of type $E_6$. Then $\Ac$ is MAT-free by
Example \ref{Exam_rk2_boolean_WeylMAT}(2).
Let $H \in \Ac$. A simple calculation (with the computer) shows that $\Ac^H$ is not MAT2-free.
\end{example}

We have two simple necessary conditions for MAT-freeness respectively MAT2-freeness. The first one is:
\begin{lemma}\label{Lem_AResnMATFree}
Let $\Ac$ be a non-empty MAT2-free arrangement with exponents $\exp(\Ac)$ $= (d_1,\ldots,d_\ell)_\leq$.
Then there is an $H \in \Ac$ such that $\vert \Ac \vert - \vert \Ac^H \vert = d_\ell$.
In particular, the same holds, if $\Ac$ is MAT-free.
\end{lemma}
\begin{proof}
By definition there are $H_q,\ldots,H_\ell \in \Ac$, $2\leq q$ such that $\Ac' := \Ac \setminus \{H_q,\ldots,H_\ell \}$
is MAT2-free.
Furthermore by condition (1) the hyperplanes $H_q,\ldots,H_\ell$ are linearly independent.
Let $H := H_\ell$. 
By condition (2), we have $X = \cap_{i=q}^\ell H_i \nsubseteq \cup_{H' \in \Ac'} H'$ 
and thus $|\Ac^{H}| = |(\Ac'\cup \{H\})^H| + \ell-q$.
Now 
$$
\vert \Ac' \vert - \vert (\Ac'\cup \{H\})^H \vert = d_\ell-1
$$ 
by condition (3) and hence
$$
\vert \Ac \vert - \vert \Ac^H \vert = |\Ac'|+\ell-q+1 - \vert (\Ac'\cup \{H\})^H \vert - \ell+ q = d_\ell.
$$
\end{proof}

The second one is:
\begin{lemma}\label{Lem_MAT_free_filtration}
Let $\Ac$ be an MAT2-free arrangement. Then $\Ac$ has a free filtration, i.e.\ $\Ac$ is additionally free.
In particular, the same is true, if $\Ac$ is MAT-free.
\end{lemma}
\begin{proof}
Let $\Ac$ be MAT2-free. Then by definition there are $H_q,\ldots,H_\ell \in \Ac$
such that $\Ac' := \Ac \setminus \{H_q,\ldots,H_\ell\}$ is MAT2-free and conditions (1)--(3) are satisfied. 
Set $\Bc := \{H_q,\ldots,H_\ell\}$. 
By \cite[Cor.~3.2]{AbeTer18_MultAddDelRes} for all $\Cc \subseteq \Bc$ the arrangement $\Ac' \cup \Cc$ is free.
Hence by induction $\Ac$ has a free filtration.
\end{proof}

\subsection*{An MAT2-free but not MAT-free arrangement}

We now provide an example of an arrangement which is MAT2-free but not MAT-free.

\begin{example}\label{Exam_MAT2bnMAT}
Let $\Ac$ be the arrangement defined by
\begin{align*}
\Ac := \{ &H_1,\ldots,H_{10} \} \\
    := \{ &(1,0,0)^\perp, (0,1,0)^\perp, (0,0,1)^\perp, (1,1,0)^\perp, (1,2,0)^\perp, (0,1,1)^\perp, \\
        &(1,3,0)^\perp, (1,1,1)^\perp, (2,3,0)^\perp, (1,3,1)^\perp \}.
\end{align*}
It is not hard to see that $\Ac$ is inductively free (actually supersolvable) with $\exp(\Ac) = (1,4,5)$.
\end{example}

\begin{proposition}\label{Prop_ExamMAT2}
The arrangement $\Ac$ from Example \ref{Exam_MAT2bnMAT} is MAT2-free.
\end{proposition}
\begin{proof}
Let $\Bc_1=\{H_1,H_2,H_3\}$, $\Bc_2=\{H_4\}$, $\Bc_3=\{H_5,H_6\}$,
$\Bc_4=\{H_7,H_8\}$, $\Bc_5 = \{H_9,H_{10}\}$, and $\Ac_k = \cup_{i=1}^k \Bc_i$ for $1 \leq k \leq 5$.
It is clear that $\Ac_1$ is MAT2-free. A simple linear algebra computation shows that the addition
of $\Bc_{i+1}$ to $\Ac_{i}$ for $1 \leq i \leq 4$ satisfies Condition (1)--(3) of Definition \ref{Def_MAT2free}.
Hence $\Ac = \Ac_5$ is MAT2-free.
\end{proof}

\begin{proposition}\label{Prop_ExamNotMAT}
The arrangement $\Ac$ from Example \ref{Exam_MAT2bnMAT} is not MAT-free.
\end{proposition}
\begin{proof}
Suppose $\Ac$ is MAT-free and $\pi = (\pi_1,\ldots,\pi_5)$ is an MAT-partition.
Since $\exp(\Ac) = (1,4,5)$ the last block $\pi_5$ has to be a singleton, i.e.\ $\pi_5 = \{H\}$. 
By Condition (3) of Lemma \ref{Lem_MATPart} we have $|\Ac^H| = 5$ and the only hyperplane with this property is $H_9 = (2,3,0)^\perp$.
Similarly $\pi_4$ can only contain one of $H_3,H_6,H_8,H_{10}$. But looking at their intersections
we see that all of the latter are contained in another hyperplane of $\Ac$, e.g.\
$H_3 \cap H_8 \subseteq H_4$. This contradicts Condition (2).
Hence $\Ac$ is not MAT-free.
\end{proof}

As a direct consequence we get:
\begin{proposition}\label{Propo_MATproperSubclassMAT2}
We have
\[
\MATF \subsetneq \MATF'.
\]
\end{proposition}

\subsection*{Products of MAT-free and MAT2-free arrangements}

As for freeness in general (Proposition \ref{Prop_ProdFree}), 
the product construction is compatible with the notion of MAT-freeness:
\begin{theorem}\label{Thm_ProdMAT}
Let $\Ac = \Ac_1 \times \Ac_2$ be a product of two arrangements.
Then $\Ac \in \MATF$ if and only if $\Ac_1 \in \MATF$ and $\Ac_2 \in \MATF$.
\end{theorem}
\begin{proof}
Assume $\Ac_i$ is an arrangement in the vector space $V_i$ of dimension $\ell_i$ for $i=1,2$.
We argue by induction on $|\Ac|$. 
If $|\Ac|=0$, i.e.\ $\Ac_1 = \Phi_{\ell_1}$, and $\Ac_2 = \Phi_{\ell_2}$ 
then the statement is clear.
Assume $\Ac_1$ is MAT-free with $\exp(\Ac_1) = (d^1_1,\ldots,d^1_{\ell_1})_\le$ and 
$\Ac_2$ is MAT-free with $\exp(\Ac_1) = (d^2_1,\ldots,d^2_{\ell_2})_\le$.
Then without loss of generality $d:=d^1_{\ell_1} \geq d^2_{\ell_2}$. 
Let $q_i$ be the multiplicity of the exponent $d$
in $\exp(\Ac_i)$ for $i=1,2$ (note that $q_2=0$ if $d>d^2_{\ell_2}$).
Then since $\Ac_i$ is MAT-free there are hyperplanes $\{H^i_1,\ldots,H^i_{q_i}\} \subseteq \Ac_i$
such that $\Ac_i' := \Ac_i \setminus \{H^i_1,\ldots,H^i_{q_i}\}$ is MAT-free, i.e.\ they satisfy Conditions (1)--(3)
from Definition \ref{Def_MATfree}.
Now by the induction hypothesis $\Ac' = \Ac_1' \times \Ac_2'$ is MAT-free and clearly 
$\{H^1_1\oplus V_2,\ldots,H^1_{q_1}\oplus V_2\} \cup \{V_1\oplus H^2_1,\ldots,V_1\oplus H^2_{q_2}\}$ satisfy Conditions (1)--(3).
Hence $\Ac$ is MAT-free.

Conversely assume $\Ac$ is MAT-free with $\exp(\Ac)=(d_1,\ldots,d_\ell)_\leq$. 
By Proposition \ref{Prop_ProdFree} both factors $\Ac_1$ and $\Ac_2$ are free with
$\exp(\Ac_i) = (d^i_1,\ldots,d^i_{\ell_i})_\leq$ and without loss of generality
$d_\ell=d^1_{\ell_1}\geq d^2_{\ell_2}$. Assume further that $q_i$ is the multiplicity
of $d_\ell$ in $\exp(\Ac_i)$ and $q$ is the multiplicity of $d_\ell$ in $\exp(\Ac)$, i.e.\ $q = q_1+q_2$.
There are hyperplanes $\{H_1,\ldots,H_q\} \subset \Ac$ such that $\Ac' = \Ac \setminus \{H_1,\ldots,H_q\}$
is MAT-free with $\exp(\Ac') = (d_1,\ldots,d_{\ell-q},d_{\ell-q+1}-1,\ldots,d_\ell-1)_\le$, 
and Conditions (1)--(3) are satisfied.
We may further assume that $H_i = H^1_i \oplus V_2$ for $1\leq i \leq q_1$ and
$H_j = V_1 \oplus H^2_{j-q_1}$ for $q_1+1 \leq j \leq q$.
Let $\Ac_i' = \Ac_i \setminus \{H^i_1,\ldots,H^i_{q_i}\}$ for $i=1,2$. Note that if $d_\ell > d^2_{\ell_2}$
we have $q_2 = 0$ and $\Ac_2' = \Ac_2$. But at least we have $\Ac_1' \subsetneq \Ac_1$.
Then $\Ac' = \Ac_1' \times \Ac_2'$, $|\Ac'| < |\Ac|$ and by the induction hypothesis $\Ac_1'$ and
$\Ac_2'$ are MAT-free and Conditions (1) and (2) are clearly satified for $\Ac_i'$
and $\{H^i_1,\ldots,H^i_{q_i}\}$.
But since 
\begin{align*}
d_\ell-1 =\, &|\Ac'| - |(\Ac'\cup\{H_i\})^{H_i}| \\
  =\, &|\Ac_1'| + |\Ac_2'| - (|(\Ac_1\cup \{H^1_i\})^{H^1_i}| + |\Ac_2'|) \\
  =\, &|\Ac_1'| - |(\Ac_1\cup \{H^1_i\})^{H^1_i}|
\end{align*}
for $1\leq i \leq q_1$ and
\begin{align*}
d_\ell-1 =\, &|\Ac'| - |(\Ac'\cup\{H_j\})^{H_j}| \\
  =\, &|\Ac_1'| + |\Ac_2'| - (|(\Ac_1\cup \{H^2_{j-q_1}\})^{H^2_{j-q_1}}| + |\Ac_2'|) \\
  =\, &|\Ac_1'| - |(\Ac_1\cup \{H^2_{j-q_1}\})^{H^2_{j-q_1}}|
\end{align*}
for $q_1+1 \leq j \leq q_2$,
Condition (3) is also satisfied for $\Ac_1'$ and $\Ac_2'$.
Hence both $\Ac_1$ and $\Ac_2$ are MAT-free.
\end{proof}
Altenatively, one can prove Theorem \ref{Thm_ProdMAT} by observing that MAT-Partitions
for $\Ac_1$ and $\Ac_2$ are directly obtained from an MAT-Partition for $\Ac$: take the non-empty factors
of each block in the same order, and vise versa: take the products of the blocks of partitions for $\Ac_1$
and $\Ac_2$.

\begin{remark}\label{Rem_ReducibleMAT}
Thanks to the preceding theorem, our classification of MAT-free irreducible reflection arrangements
proved in the next 2 sections gives actually a classification of all MAT-free reflection arrangements:
a reflection arrangement $\Ac(W)$ is MAT-free if and only if it has no irreducible factor isomorphic to
one of the non-MAT-free irreducible reflection arrangements listed in Theorem \ref{Thm_matref}.
\end{remark}

In contrast to MAT-freeness, the weaker notion of MAT2-freeness is not compatible with products
as the following example shows:
\begin{example}
Let $\Ac_1$ be the MAT2-free but not MAT-free arrangement of Example \ref{Exam_MAT2bnMAT}
with exponents $\exp(\Ac_1) = (1,4,5)$. Let $\zeta = \frac{1}{2}(-1+i\sqrt{3})$ be a primitive cube root of unity,
and let $\Ac_2$ be the arrangement defined by the following linear forms:
\begin{align*}
\Ac_2 := \{ &H^2_1,\ldots,H^2_{10}\} \\
    := \{ &(1,0,0)^\perp,(0,1,0)^\perp,(0,0,1)^\perp,(1,-\zeta,0)^\perp,(1,0,-\zeta)^\perp \\
        &(1,-\zeta^2,0)^\perp,(1,0,-\zeta^2)^\perp,(1,-1,0)^\perp,(1,0,-1)^\perp, (0,1,-\zeta)^\perp \}.
\end{align*}
A linear algebra computation shows that $\pi = (1,2,3|4,5|6,7|8,9|10)$ is an MAT-partition for $\Ac_2$.
In particular $\Ac_2$ is MAT2-free with $\exp(\Ac_2) = (1,4,5)$.

Now by Proposition \ref{Prop_ProdFree} the product $\Ac := \Ac_1 \times \Ac_2$ is
free with $\exp(\Ac) =(1,1,$ $4,4,5,5)$. Suppose $\Ac$ is MAT2-free.
Then either there are hyperplanes $H_1 \in \Ac_1$ and $H_2 \in \Ac_2$ 
such that $\Ac' = \Ac_1' \times \Ac_2'$ is MAT2-free with exponents $\exp(\Ac')=(1,1,4,4,4,4)$ 
where $\Ac_i' = \Ac_i \setminus \{H_i\}$.
Or there are hyperplanes $H^1_1,H^1_2 \in \Ac_1$, $H^2_1,H^2_2\in \Ac_2$ 
such that $\Ac' = \Ac_1' \times \Ac_2'$ is MAT2-free with exponents $\exp(\Ac')=(1,1,3,3,4,4)$ 
where $\Ac_i' = \Ac_i \setminus \{H^i_1,H^i_2\}$.

In the first case $\Ac'$ is actually MAT-free by Remark \ref{Rem_MATexpMAT2}. 
But then by Theorem \ref{Thm_ProdMAT} $\Ac_2'$ is MAT-free and $\Ac_2$ is
MAT-free too which is a contradiction.

In the second case
$H^1_1\oplus V_2,H^1_2\oplus V_2, V_1 \oplus H^2_1, V_1 \oplus H^2_2$ satisfy Condition (1)--(3) 
of Defintion \ref{Def_MAT2free}. But by Condition (3) we have
\[
|\Ac_1'| - |(\Ac_1'\cup\{H^1_1\})^{H^1_1}| = 4
\]
and
\[
|\Ac_1'| - |(\Ac_1'\cup\{H^1_2\})^{H^1_2}| = 3.
\]
But an easy calculation shows that there are no two hyperplanes in $\Ac_1$ with this property and
which also satisfy Condition (2)--(3). 
This is a contradiction and hence $\Ac = \Ac_1 \times \Ac_2$ is not MAT2-free.
\end{example}


\section{MAT-free imprimitive reflection groups}\label{Sec_ProofImprim}

\begin{defn}[{\cite[\S 6.4]{OrTer92_Arr}}]
Let $x_1,\ldots,x_\ell$ be a basis of $V^*$. Let $\zeta = \exp(\frac{2\pi i}{r})$ ($r \in \NN$) be a primitive $r$-th root of unity.
Define the linear forms $\alpha_{ij}(\zeta^k) \in V^*$ by
$$
\alpha_{ij}(\zeta^k) = x_i - \zeta^k x_j
$$
and the hyperplanes
$$
H_{ij}(\zeta^k) = \ker(\alpha_{ij}(\zeta^k)).
$$
for $1 \leq i,j \leq \ell$ and $1 \leq k \leq r$.
Then the reflection arrangement of the imprimitive complex reflection group $G(r,1,\ell)$ can be defined by:
$$
\Ac(G(r,1,\ell)) = \{\ker(x_i) \mid 1 \leq i \leq \ell\} \dot{\cup} \{ H_{ij}(\zeta^k) \mid 1\leq i < j \leq \ell,\, 1\leq k \leq r\}.
$$
\end{defn}

\begin{proposition}\label{Prop_Monr1lMATFree}
Let $\Ac=\Ac(G(r,1,\ell))$. Let
$$
\pi_{11} := \{\ker(x_i) \mid 1 \leq i \leq \ell \},
$$
and
$$
\pi_{ij} := \{H_{(i-1)k}(\zeta^j) \mid i \leq k \leq \ell \},
$$
for $2\leq i \leq \ell$, $1\leq j \leq r$.
Then
\begin{align*}
\pi = \, &(\pi_{ij})_{\substack{1 \leq i \leq \ell, \\ 1\leq j \leq m_i}},\
    m_i =  \begin{cases}
         1  & \quad \text{for } i=1\\
    r  & \quad \text{for } 2\leq i \leq \ell
  \end{cases} \\
  = \, &(\pi_{11}|\pi_{21}|\cdots|\pi_{2r}|\cdots|\pi_{\ell r})
\end{align*}
is an MAT-partition of $\Ac$.
In particular $\Ac \in \MATF$ with exponents
$$
\exp(\Ac) = (1,r+1,2r+1,\ldots,(l-1)r +1).
$$
\end{proposition}
\begin{proof}
We verify Conditions (1)--(3) from Lemma \ref{Lem_MATPart} in turn.

Let
$$
\Ac_{ij} := (\bigcup_{\substack{1 \leq a \leq i-1, \\ 1 \leq b \leq m_a}} \pi_{ab}) \cup (\bigcup_{1\leq b \leq j} \pi_{ib})
$$
and
$$
\Ac_{ij}' := (\bigcup_{\substack{1 \leq a \leq i-1, \\ 1 \leq b \leq m_a}} \pi_{ab}) \cup (\bigcup_{1\leq b \leq j-1} \pi_{ib}).
$$

For $\pi_{11}$ we clearly have $\vert \pi_{11} \vert = \rk(\pi_{11}) = \ell$.
Similarly for $2 \leq i \leq \ell$, $1 \leq j \leq r$ we have 
$\vert \pi_{ij} \vert = \rk(\pi_{ij}) = \ell-i+1$ since all the defining linear forms $\alpha_{(i-1)k}(\zeta^j)$ ($i \leq k \leq \ell$) 
for the hyperplanes in $\pi_{ij}$ are linearly independent. Thus Condition (1) holds.

Furthermore, the forms $\{\alpha_{ac}(\zeta^b)\} \dot{\cup} \{\alpha_{(i-1)k}(\zeta^j) \mid i \leq k \leq \ell \}$ are linearly independent
for all $1 \leq a \leq i-1$, $1 \leq b \leq j-1$, and $a+1 \leq c \leq \ell$, i.e.\ $\cap_{H \in \pi_{ij}}H =: X_{ij} \nsubseteq H$ for all
$H \in \Ac_{ij}'$.
Hence Condition (2) is also satisfied.

To verify Condition (3) let $H=H_{(i-1)k}(\zeta^j) \in \pi_{ij}$ for a fixed $1 \leq k \leq r$. We show
$$
\vert \Ac_{ij}'\vert - (j+(i-2)r) = \vert(\Ac_{ij}')^H \vert.
$$
Let $H_a' := H_{(i-1)k}(\zeta^a) \in \Ac_{ij}'$, $1 \leq a \leq j-1$. Then
$$
\Bc := (\Ac_{ij}')_{H\cap H_a'} = \{\ker(x_{i-1}),\ker(x_k)\} \dot{\cup} \{ H'_b \mid 1 \leq b \leq j-1 \},
$$
and $\rk(\Bc)=2$. So all $H' \in \Bc$ give the same intersection with $H$ and $\vert \Bc \vert = j + 1$.
For $H' = H_{a(i-1)}(\zeta^b) \in \Ac_{ij}'$ with $a \leq i-2$, and $1 \leq b \leq r$ we have
$$
\Cc := (\Ac_{ij}')_{H\cap H'} = \{H', H_{ak}(\zeta^(j+b))\},
$$
$\vert \Cc \vert = 2$ and there are exactly $(i-2)r$ such $H'$.
All other $H'' \in \Ac_{ij}'$ intersect $H$ simply.
Hence
\begin{align*}
\vert (\Ac_{ij}')^H) \vert &= \vert \Ac_{ij}' \vert - (|\Bc|-1) - (i-2)r(|\Cc|-1) \\
&= \vert \Ac_{ij}' \vert - j - (i-2)r,
\end{align*}
or $\vert \Ac_{ij}' \vert - \vert (\Ac_{ij}')^H) \vert = \sum_{a=1}^{i-1} m_i + (j-1)$.
This finishes the proof.
\end{proof}

\begin{proposition}\label{Prop_MoneelNMATFree}
Let $\Ac = \Ac(G(r,r,\ell))$ ($r, \ell \geq 3$). Then $\Ac$ is not MAT2-free.
In particular $\Ac$ is not MAT-free.
\end{proposition}
\begin{proof}
By \cite[Prop.~6.85]{OrTer92_Arr} the arrangement $\Ac$ is free with $\exp(\Ac) = (d_1,\ldots,d_{\ell}) = (1,r+1,2r+1,\ldots,(\ell-2)r+1,(\ell-1)(r-1))$. In particular
we have 
$(\ell-1)(r-1) = d_\ell$ and $\vert \Ac \vert  = \frac{\ell(\ell-1)}{2}r$.
But for all $H \in \Ac$ by \cite[Prop.~6.82, 6.85]{OrTer92_Arr} we have $\vert \Ac^H \vert = \frac{(\ell-1)(\ell-2)}{2}r +1$.
Hence $\vert \Ac \vert - \vert \Ac^H \vert = (\ell-1)r -1 \neq d_\ell$ and by Lemma \ref{Lem_AResnMATFree} the arrangement $\Ac$ is not MAT2-free.
\end{proof}

\begin{theorem}
Let $\Ac = \Ac(W)$ be the reflection arrangement of the imprimitive complex reflection group $W = G(r,e,\ell)$ ($r, \ell \geq 3$).
Then $\Ac$ is MAT-free if and only if it is MAT2-free if and only if $e\neq r$.
\end{theorem}
\begin{proof}
Since $\Ac = \Ac(G(r,1,\ell))$ if and only if $r\neq e$,
this is Proposition \ref{Prop_Monr1lMATFree} and Proposition \ref{Prop_MoneelNMATFree}.
\end{proof}


\section{MAT-free exceptional complex reflection groups}\label{Sec_ProofPrim}
To prove the MAT-freeness of one of the following reflection arrangements, we explicitly give a realization by linear forms.

First note that if $W$ is an exceptional Weyl group, or a group of rank $\leq2$, 
then by Example \ref{Exam_rk2_boolean_WeylMAT} $\Ac(W)$ is MAT-free.

\begin{proposition}
Let $\Ac$ be the reflection arrangement of the reflection group $H_3$ (Shephard-Todd: $G_{23}$).
Then $\Ac$ is MAT-free. In particular $\Ac$ is MAT2-free.
\end{proposition}
\begin{proof}
Let $\tau = \frac{1+\sqrt{5}}{2}$ be the golden ratio and $\tau' = 1/\tau$ its reciprocal.
The arrangement $\Ac$ can be defined by the following linear forms:
\begin{align*}
  \Ac = \{ &H_1,\ldots,H_{15} \} \\
  = \{ &(1,0,0)^\perp, (0,1,0)^\perp, (0,0,1)^\perp, (1,\tau,\tau')^\perp, (\tau',1,\tau)^\perp, (\tau,\tau',1)^\perp, \\
  &(1,-\tau,\tau')^\perp, (\tau',1,-\tau)^\perp, (-\tau,\tau',1)^\perp, (1,\tau,-\tau')^\perp, (-\tau',1,\tau)^\perp, \\
  &(\tau,-\tau',1)^\perp, (1,-\tau,-\tau')^\perp, (-\tau',1,-\tau)^\perp, (-\tau,-\tau',1)^\perp \}.
\end{align*}
With this linear ordering of the hyperplanes the partition
$$
\pi = (13,14,15|10,12|5,6|4,11|8,9|7|3|2|1)
$$
satisfies Conditions (1)--(3) of Lemma \ref{Lem_MATPart} as one can verify by an easy linear algebra computation.
Hence $\pi$ is an MAT-partition and $\Ac$ is MAT-free.
\end{proof}

\begin{proposition}
Let $\Ac$ be the reflection arrangement of the complex reflection group $G_{24}$. Then $\Ac$ is not MAT2-free. In particular $\Ac$ is not MAT-free.
\end{proposition}
\begin{proof}
The arrangement $\Ac$ is free with $\exp(\Ac)=(1,9,11)$ and $|\Ac|-|\Ac^H| = 13$ for all $H \in \Ac$ by \cite[Tab.~C.5]{OrTer92_Arr}.
Hence by Lemma \ref{Lem_AResnMATFree} $\Ac$ is not MAT2-free.
\end{proof}

\begin{proposition}
Let $\Ac$ be the reflection arrangement of the complex reflection group $G_{25}$.
Then $\Ac$ is MAT-free. In particular $\Ac$ is MAT2-free.
\end{proposition}
\begin{proof}
Let $\zeta = \frac{1}{2}(-1+i\sqrt{3})$ be a primitive cube root of unity. The reflecting hyperplanes of $\Ac$ can be defined by the following
linear forms (cf.\ \cite[Ch.\ 8, 5.3]{LehTay09_UnReflGrps}):
\begin{align*}
  \Ac = \{ &H_1,\ldots,H_{12} \} \\
  = \{ &(1,0,0)^\perp, (0,1,0)^\perp, (0,0,1)^\perp, (1,1,1)^\perp, (1,1,\zeta)^\perp, (1,1,\zeta^2)^\perp, \\
  &(1,\zeta,1)^\perp, (1,\zeta,\zeta)^\perp, (1,\zeta,\zeta^2)^\perp, (1,\zeta^2,1)^\perp, (1,\zeta^2,\zeta)^\perp, (1,\zeta^2,\zeta^2)^\perp \}.
\end{align*}
With this linear ordering of the hyperplanes the partition
$$
\pi = (7,4,3|8,5|9,6|2,1|10|11|12)
$$
satisfies the three conditions of Lemma \ref{Lem_MATPart} as one can easily verify by a linear algebra computation.
Hence $\pi$ is an MAT-partition and $\Ac$ is MAT-free.
\end{proof}

\begin{proposition}
Let $\Ac$ be the reflection arrangement of the complex reflection group $G_{26}$.
Then $\Ac$ is MAT-free.  In particular $\Ac$ is MAT2-free.
\end{proposition}
\begin{proof}
  Let $\zeta = \frac{1}{2}(-1+i\sqrt{3})$ be a primitive cube root of unity.
  The reflection arrangement $\Ac$ is the union of the reflecting hyperplanes of $\Ac(G_{25})$ and $\Ac(G(3,3,3))$ (cf.\ \cite[Ch.~8, 5.5]{LehTay09_UnReflGrps}).
  In particular the hyperplanes contained in $\Ac$ can be defined by the following linear forms:
  \begin{align*}
    \Ac = \{ &H_1,\ldots,H_{21} \} \\
    = \{ &(1,0,0)^\perp, (0,1,0)^\perp, (0,0,1)^\perp, (1,1,1)^\perp, (1,1,\zeta)^\perp, (1,1,\zeta^2)^\perp, \\
    &(1,\zeta,1)^\perp, (1,\zeta,\zeta)^\perp, (1,\zeta,\zeta^2)^\perp, (1,\zeta^2,1)^\perp, (1,\zeta^2,\zeta)^\perp, (1,\zeta^2,\zeta^2)^\perp, \\
    &(1,-\zeta,0)^\perp, (1,-\zeta^2,0)^\perp, (1,-1,0)^\perp, (1,0,-\zeta)^\perp, (1,0,-\zeta^2)^\perp, \\
    &(1,0,-1)^\perp, (0,1,-\zeta)^\perp, (0,1,-\zeta^2)^\perp, (0,1,-1)^\perp \}.
  \end{align*}
  With this linear ordering of the hyperplanes the partition
  $$
  \pi = (12,19,20|16,18|13,15|17,21|10,14|6,11|8,9|7|5|4|3|2|1)
  $$
  satisfies the three conditions of Lemma \ref{Lem_MATPart} as one can verify by a standard linear algebra computation.
  Hence $\pi$ is an MAT-partition and $\Ac$ is MAT-free.
\end{proof}

\begin{proposition}
Let $\Ac$ be the reflection arrangement of the complex reflection group $G_{27}$. 
Then $\Ac$ is not MAT2-free.  In particular $\Ac$ is not MAT-free.
\end{proposition}
\begin{proof}
The arrangement $\Ac$ is free with $\exp(\Ac)=(1,19,25)$ and $|\Ac|-|\Ac^H| = 29$ for all $H \in \Ac$ by \cite[Tab.~C.8]{OrTer92_Arr}.
Hence by Lemma \ref{Lem_AResnMATFree} $\Ac$ is not MAT2-free.
\end{proof}

\begin{proposition}
Let $\Ac$ be the reflection arrangement of the reflection group $H_4$ (Shephard-Todd: $G_{30}$).
Then $\Ac$ is MAT-free. In particular $\Ac$ is MAT2-free.
\end{proposition}
\begin{proof}
Let $\tau = \frac{1+\sqrt{5}}{2}$ be the golden ratio and $\tau' = 1/\tau$ its reciprocal.
The arrangement $\Ac$ can be defined by the following linear forms:
\begin{align*}
  \Ac = \{ &H_1,\ldots,H_{60} \} \\
  = \{ &(1,0,0,0)^\perp, (0,1,0,0)^\perp, (0,0,1,0)^\perp, (0,0,0,1)^\perp, (1,\tau,\tau',0)^\perp, \\
&(1,0,\tau,\tau')^\perp, (1,\tau',0,\tau)^\perp, (\tau,1,0,\tau')^\perp, (\tau',1,\tau,0)^\perp, (0,1,\tau',\tau)^\perp, \\
&(\tau,\tau',1,0)^\perp, (0,\tau,1,\tau')^\perp, (\tau',0,1,\tau)^\perp, (\tau,0,\tau',1)^\perp, (\tau',\tau,0,1)^\perp, \\
&(0,\tau',\tau,1)^\perp, (-1,\tau,\tau',0)^\perp, (1,-\tau,\tau',0)^\perp, (1,\tau,-\tau',0)^\perp, (-1,0,\tau,\tau')^\perp, \\
&(1,0,-\tau,\tau')^\perp, (1,0,\tau,-\tau')^\perp, (-1,\tau',0,\tau)^\perp, (1,-\tau',0,\tau)^\perp, (1,\tau',0,-\tau)^\perp, \\
&(-\tau,1,0,\tau')^\perp, (\tau,-1,0,\tau')^\perp, (\tau,1,0,-\tau')^\perp, (-\tau',1,\tau,0)^\perp, (\tau',-1,\tau,0)^\perp, \\
&(\tau',1,-\tau,0)^\perp, (0,-1,\tau',\tau)^\perp, (0,1,-\tau',\tau)^\perp, (0,1,\tau',-\tau)^\perp, (-\tau,\tau',1,0)^\perp, \\
&(\tau,-\tau',1,0)^\perp, (\tau,\tau',-1,0)^\perp, (0,-\tau,1,\tau')^\perp, (0,\tau,-1,\tau')^\perp, (0,\tau,1,-\tau')^\perp, \\
&(-\tau',0,1,\tau)^\perp, (\tau',0,-1,\tau)^\perp, (\tau',0,1,-\tau)^\perp, (-\tau,0,\tau',1)^\perp, (\tau,0,-\tau',1)^\perp, \\
&(\tau,0,\tau',-1)^\perp, (-\tau',\tau,0,1)^\perp, (\tau',-\tau,0,1)^\perp, (\tau',\tau,0,-1)^\perp, (0,-\tau',\tau,1)^\perp, \\
&(0,\tau',-\tau,1)^\perp, (0,\tau',\tau,-1)^\perp, (1,1,1,1)^\perp, (-1,1,1,1)^\perp, (1,-1,1,1)^\perp, \\
&(1,1,-1,1)^\perp, (1,1,1,-1)^\perp, (-1,-1,1,1)^\perp, (-1,1,-1,1)^\perp, (-1,1,1,-1)^\perp \}.
\end{align*}
With this linear ordering of the hyperplanes the partition
\begin{align*}
\pi =
( &\,31,  43,  48,  54
  |  29,  38,  51
  |  23,  34,  58
  |  18,  20,  25
  |  17,  59,  60 \\
 &|  21,  47,  52
  |  39,  41,  44
  |  26,  32,  49
  |  30,  35,  40
  |   2,   3,  42
  |  33,  46,  50 \\
 &|   4,  37
  |  27,  57
  |  19,  24
  |  55,  56
  |  10,  22
  |  12,  45
  |  16,  28
  |  15,  36 \\
 &|  53
  |  14
  |  13
  |  11
  |   9
  |   8
  |   7
  |   6
  |   5
  |   1 )
\end{align*}
satisfies Conditions (1)--(3) of Lemma \ref{Lem_MATPart} as one can verify with a linear algebra computation.
Hence $\pi$ is an MAT-partition and $\Ac$ is MAT-free. In particular $\Ac$ is MAT2-free.
\end{proof}

We recall the following result about free filtration subarrangements of $\Ac(G_{31})$:
\begin{proposition}[{\cite[Pro.~3.8]{Mue17_RecFreeRef}}]
  Let $\Ac := \Ac(G_{31})$ be the reflection arrangement of the finite complex reflection group $G_{31}$.
  Let $\tilde{\Ac}$ be a minimal (w.r.t. the number of hyperplanes) free filtration subarrangement.
  Then $\tilde{\Ac} \cong \Ac(G_{29})$.
\end{proposition}

\begin{corollary}\label{Coro_G31_no_free_filtration}
  Let $\Ac$ be the reflection arrangement of one of the complex reflection groups $ G_{29}$ or $G_{31}$.
  Then $\Ac$ has no free filtration.
\end{corollary}

\begin{proposition}
Let $\Ac$ be the reflection arrangement of one of the complex reflection groups $G_{29}$ or $G_{31}$. Then $\Ac$ is not MAT2-free.
In particular $\Ac$ is not MAT-free.
\end{proposition}
\begin{proof}
By Corollary \ref{Coro_G31_no_free_filtration} both arrangements have no free filtration and hence are not MAT2-free by Lemma \ref{Lem_MAT_free_filtration}.
\end{proof}

\begin{proposition}
Let $\Ac$ be the reflection arrangement of the complex reflection group $G_{32}$. Then $\Ac$ is not MAT-free
and also not MAT2-free.
\end{proposition}
\begin{proof}
Up to symmetry of the intersection lattice there are exactly 9 different choices of a basis, where a basis is a 
subarrangement $\Bc \subseteq \Ac$ with $|\Bc|=r(\Bc)=r(\Ac)=4$. 
Suppose that $\Ac$ is MAT-free. Then the first block in an MAT-partition for $\Ac$ has to be one of these bases.
But a computer calculation shows that non of these bases may be extended to an MAT-partition for $\Ac$.
Hence $\Ac$ is not MAT-free.
A similar but more cumbersome calculation shows that $\Ac$ is also not MAT2-free.
\end{proof}

\begin{proposition}
Let $\Ac$ be the reflection arrangement of one of the complex reflection group $G_{33}$ or $G_{34}$. Then $\Ac$ is not MAT2-free.
In particular $\Ac$ is not MAT-free.
\end{proposition}
\begin{proof}
First, let $\Ac = \Ac(G_{33})$. Then $\exp(\Ac) = (1,7,9,13,15)$ by \cite[Tab.~C.14]{OrTer92_Arr}. But $|\Ac|-|\Ac^H| = 17$ for all $H \in \Ac$ also by
\cite[Tab.~C.14]{OrTer92_Arr}. So $\Ac$ is not MAT2-free by Lemma \ref{Lem_AResnMATFree}.

Similarly $\Ac = \Ac(G_{34})$ is free with $\exp(\Ac) = (1,13,19,25,31,37)$ by \cite[Tab.~C.17]{OrTer92_Arr} and $|\Ac|-|\Ac^H| = 41$ for all $H \in \Ac$.
Hence $\Ac$ is not MAT2-free by Lemma \ref{Lem_AResnMATFree}.
\end{proof}

Comparing with Theorem \ref{Thm_ClassIFReflArr}
finishes the proofs of Theorem \ref{Thm_matref} and Theorem \ref{Thm_mat2free}.


\section{Further remarks on MAT-freeness}\label{Sec_Remarks}

In their very recent note \cite{HogeRoehrle19_ConjAbeAF} Hoge and R\"ohrle confirmed a conjecture by Abe \cite{Abe18_AddDel_Combinatorics}
by providing two examples $\Bc$, $\Dc$ of arrangements, related to the exceptional reflection arrangement $\Ac(E_7)$,
which are additionally free but not divisionally free and in particular also not inductively free. 
The arrangements have exponents $\exp(\Bc) = (1,5,5,5,5,5,5)$ and $\exp(\Dc) = (1,5,5,5,5)$.
Since both arrangements have only 2 different exponents by Remark \ref{Rem_MATexpMAT2} they are MAT-free if and only if they are MAT2-free.
Now a computer calculation shows that both arrangements are not MAT-free and hence also not MAT2-free.  
In particular they provide no negative answer to Question \ref{Ques_MATIF} and Question \ref{Ques_MATDF}.

Several computer experiments suggest that 
similar to the poset obtained from the positive roots of a Weyl group 
giving rise to an MAT-partition (cf.\ Example \ref{Exam_rk2_boolean_WeylMAT})
MAT-free arrangements might in general satisfy a certain poset structure:
\begin{problem}\label{Prob_MATPoset}
Can MAT-freeness be characterized by the existence of a partial order on the hyperplanes, 
generalizing the classical partial order on the positive roots of a Weyl group?
\end{problem} 

Recall that by Example \ref{Exam_ResNotMAT2} the restriction $\Ac^H$ is in general not
MAT-free (MAT2-free) if the arrangement $\Ac$ is MAT-free (MAT2-free).
But regarding localizations there is the following:
\begin{problem}\label{Prob_MATLocal}
Is $\Ac_X$ MAT-free (MAT2-free) for all $X \in L(\Ac)$ provided 
$\Ac$ is MAT-free (MAT2-free)?
\end{problem} 

Last but not least, related to the previous problem, our investigated examples suggest the following:
\begin{problem}\label{Prob_MaxExpLocal}
Suppose $\Ac'$ and $\Ac = \Ac' \cup \{H\}$ are free arrangements
such that $\exp(\Ac') = (d_1,\ldots,d_\ell)_\le$ and $\exp(\Ac) = (d_1,\ldots,d_{\ell-1},d_\ell+1)_\le$.
Let $X \in L(\Ac)$ with $X \subseteq H$.
By \cite[Thm.~4.37]{OrTer92_Arr} both localizations $\Ac'_X$ and $\Ac_X$ are free.
If $\exp(\Ac'_X) = (c_1,\ldots,c_r)_\le$ is it true that $\exp(\Ac) = (c_1,\ldots,c_{r-1},c_r+1)_\le$,
i.e.~ if we only increase the highest exponent is the same true for all localizations? 
\end{problem}
Note that the answer is yes if we only look at localizations of rank $\leq 2$.
Our proceeding investigation of Problem \ref{Prob_MATPoset} suggests that this should be true at least for MAT-free arrangements.
Furthermore, a positive answer to Problem \ref{Prob_MaxExpLocal} would imply 
(with a bit more work) a positive answer to Problem \ref{Prob_MATLocal}.


\newcommand{\etalchar}[1]{$^{#1}$}
\providecommand{\bysame}{\leavevmode\hbox to3em{\hrulefill}\thinspace}
\providecommand{\MR}{\relax\ifhmode\unskip\space\fi MR }
\providecommand{\MRhref}[2]{%
  \href{http://www.ams.org/mathscinet-getitem?mr=#1}{#2}
}
\providecommand{\href}[2]{#2}

\end{document}